\documentclass[11pt]{amsart}

\usepackage[pdftex]{graphicx}
\usepackage[a4paper,margin=2.5cm]{geometry}

\usepackage{amsfonts}
\usepackage{amsmath}
\usepackage{amssymb}
\usepackage{amsthm}
\usepackage{dsfont}
\usepackage{tikz}
\usepackage{float}
\usepackage{subcaption}

\usepackage{mathtools}
\usepackage{relsize}

\usepackage{rotating}

\usepackage[T1]{fontenc}

\usepackage{paralist}
\usepackage{color}

\usepackage{overpic}

\usepackage{thmtools}
\usepackage[colorlinks,cite color=blue,pagebackref=true,pdftex]{hyperref}
\usepackage{cleveref}

\usetikzlibrary{positioning}

\newcommand\R{\mathbb{R}}
\newcommand\Z{\mathbb{Z}}
\newcommand\N{\mathbb{N}}

\newcommand{\proj}{\mathrm{proj}}

\newtheorem{theorem}{Theorem}[section]
\newtheorem{corollary}[theorem]{Corollary}
\newtheorem{lemma}[theorem]{Lemma}
\newtheorem{proposition}[theorem]{Proposition}

\theoremstyle{definition}

\newtheorem{example}[theorem]{Example}
\newtheorem{remark}[theorem]{Remark}

\title[]{Weighted lattice point enumeration in lecture hall order polytopes}

\author{Katharina Jochemko}
\author{Krishna Menon}
\address{Department of Mathematics, KTH Royal Institute of Technology, Stockholm, Sweden}
\email{\{jochemko,puzhan\}@kth.se}

\date{\today}
\thanks{}

\begin{document}

\begin{abstract}
We consider weighted lattice point enumeration in lecture hall order polytopes. We observe that these half-open polytopes are closed under weight-lifting for a large class of polynomial weights, including monomials and weights given by binomial coefficients. For rooted trees, we prove that the corresponding (weighted) $h^\ast$-polynomial is real-rooted thereby complementing previous results of Savage-Visontai and Brändén-Leander.
    
\end{abstract}

\maketitle

\section{Introduction and results}
Given a partially ordered set $P=([d],\preceq)$ on $[d]:=\{1,\ldots, d\}$, Brändén and Leander~\cite{lec_poset} introduced $s$-lecture hall $P$-partitions thereby simultaneously generalizing the classical concept of a $P$-partition developed by Stanley~\cite{enumcomb1} and $s$-lecture hall partitions pioneered by Bousquet-Mélou and Eriksson~\cite{BousquetMelouEriksson}. For $s\colon [d]\rightarrow \mathbb{N}$, a map $f\colon [d]\rightarrow \mathbb{Z}$ is called a \textbf{$s$-lecture hall $P$-partition} (or \textbf{$(P,s)$-partition} for short) if 
\begin{enumerate}
    \item[(i)] $f(i)/s(i)\leq f(j)/s(j)$ whenever $i\prec j$, and 
    \item[(ii)] $f(i)/s(i)<f(j)/s(j)$ whenever $i\prec j$ and $i>j$, 
\end{enumerate}
where $\leq$ denotes the usual total order on the real numbers. The corresponding \textbf{$s$-lecture hall order polytope} is defined as
\[
\mathcal{O}(P,s)=\{f \in \mathbb{R}^d: f \text{ is a } (P,s)\text{-partition and }0\leq f(i)/s(i)\leq 1 \text{ for all }i\in [d]\} \, .
\]
Observe that $\mathcal{O}(P,s)$ is the intersection of finitely many closed and open half-spaces and thus in general half-open. If $s$ is the all-one map then we obtain the (half-open) order polytope $\mathcal{O}(P)$ studied by Stanley~\cite{Stanleyorderchain}.

In this note we investigate structural and arithmetic aspects of weighted lattice point enumeration in $s$-lecture hall order polytopes. By a classical theorem of Ehrhart, given any lattice polytope $Q\subset \mathbb{R}^d$, the number of lattice points in nonnegative integer dilates $nQ=\{nq:q\in Q\}$, $n\geq 0$, is given by a polynomial $\mathrm{E}_Q$ in the dilation factor, that is $\mathrm{E}_Q(n)=|nQ\cap \mathbb{R}^d|$ for integers $n\geq 0$. The polynomial $\mathrm{E}_Q$ is called the \textbf{Ehrhart polynomial} of $Q$. The polynomiality extends beyond simple counting of lattice points to weighted counting if the weights are given by a polynomial~\cite{PK,Alesker}. More precisely, for a polynomial function $w\colon \mathbb{R}^d \rightarrow \mathbb{R}$, there is a polynomial $\mathrm{E}_Q(w;n)$ in~$n$, the \textbf{weighted Ehrhart polynomial}, such that
\[
\mathrm{E}_Q(w;n)  \ = \ \sum _{x\in nP\cap \mathbb{Z}^d}w(x) \, \text{ for integers}n\geq 0\, .
\]
A natural question is to identify situations in which the weighted Ehrhart polynomial is in fact equal to the (unweighted) Ehrhart polynomial of a (higher dimensional) lattice polytope. De Loera et al~\cite{SumsLatticePoints} showed that for a large class of natural  weights $w$, which include monomials and certain binomial coefficients, and any lattice polytope of the form $Q=\{x\colon Ax=b, x\geq 0\}$ there exists a polytope $Q^\ast$, called the \textbf{weight lifting polytope}, such that $\mathrm{E}_Q(w;n)=\mathrm{E}_{Q^\ast}(1;n)=\mathrm{E}_{Q^\ast}(n)$. 

In the following we consider weighted lattice point enumeration in $\mathcal{O}(P,s)$. As a first result we show that the class of $s$-lecture hall order polytopes are closed under weight lifting in the following situations; here, we use $x(k)$ to denote the $k$-th coordinate of $x$.
\begin{proposition}\label{prop:weightlifting}
    Let $\ell \geq 0$ and let $w: \R^d \rightarrow \R$ be given by $w(x) = w_1(x) \cdots w_\ell(x)$ where for each $i \in [\ell]$, the function $w_i$ is of the form
    \begin{itemize}
    \setlength{\itemsep}{0.25cm}
        \item[(i)] $w_i(x) = \dbinom{x(k) + a}{b}$ for some $k \in [d]$ and integers $0 \leq a \leq b$,
        \item[(ii)] $w_i(x) = \dbinom{\lfloor x(k)/s(k)\rfloor + a}{b}$ for some $k \in [d]$ and integers $0 < a \leq b$, or
        \item[(iii)] $w_i(x) = \dbinom{\lceil x(k)/s(k)\rceil + a}{b}$ for some $k \in [d]$ and integers $0 \leq a < b$.
    \end{itemize}
    Then there exists a poset $P^\ast$ and map $s^\ast : P^\ast \rightarrow \N$ such that \[\mathrm{E}_{\mathcal{O}(P,s)}(w;n)=\mathrm{E}_{\mathcal{O}(P^\ast,s^\ast)}(n) \, . \]
\end{proposition}
In the proof, we provide an explicit construction of $P^\ast$ and $s^\ast$ as extensions of $P$ and $s$, respectively. Observe that the above proposition covers the situation when $w$ is a monomial. Further, note that the last two types of weights are in general no longer polynomials but quasi-polynomials. The result  nevertheless guarantees polynomiality of the weighted Ehrhart function $\mathrm{E}_{\mathcal{O}(P,s)}(w,n)$ via the weight lifting polytope.

An important linear transform of the Ehrhart polynomial is the \textbf{$h^\ast$-polynomial}. For a $d$-dimensional lattice polytope $Q$ it is given as the numerator of the generating series
\[
\sum _{n\geq 0}\mathrm{E}_Q(n)t^n \ = \ \frac{h^\ast (t)}{(1-t)^{d+1}} \, .
\]
 A classical theorem by Stanley~\cite{StanleyNonnegative} guarantees that the coefficients of $h^\ast$ are nonnegative integers. Of ongoing particular interest is the investigation of other families of inequalities, such unimodality and log-concavity conditions on the coefficient sequence~\cite{FerroniHigashitani}. Both these properties follow automatically if the $h^\ast$-polynomial is real-rooted, see, e.g.,~\cite{unimodalitybeyond}. For order polytopes the real-rootedness was conjectured by the Neggers-Stanley conjecture. While generally disproven~\cite{Brandencounterexample,Stembridgecounterexample}, the conjecture is true for order polytopes of series-parallel posets due to results by Wagner~\cite{Wagnerseriesparallel}. A natural question is whether this can be extended to $s$-lecture hall order polytopes. Compare also~\cite[Section 5]{lec_poset}. Savage and Visontai~\cite{sav_vis} proved real-rootedness of the $h^\ast$-polynomial of $\mathcal{O}(P,s)$ for any nonnegative sequence $s$ in case that the poset $P$ is a chain. To prove this result they used an interpretation of the  $h^\ast$-polynomial in terms of ascent statistics over generalized inversion sequences as well as interlacing polynomials, thereby pioneering the method within Ehrhart theory. For general $s$-lecture hall order polytopes, Brändén and Leander~\cite{lec_poset} gave an interpretation of the $h^\ast$-polynomials as generating polynomial of the descent statistics over so-called $s$-colored permutations, generalizing previous results on $P$-Eulerian polynomials of Stanley~\cite{enumcomb1}. Using this interpretation they were able to prove real-rootedness for $\mathcal{O}(P,s)$ if $P$ is an ordinal sum of antichains and $s$ is constant on every antichain in the sum thereby generalizing the results in~\cite{sav_vis}. Continuing the investigations for series-parallel posets, we show that the $s$-lecture hall order polytope of rooted trees is always real-rooted. Here and in the following we always assume the rooted tree to be hanging.

 \begin{theorem}\label{thm:realrooted}
     Let $P$ be a rooted tree on $[d]$ and $s\colon [d]\rightarrow \mathbb{N}$. Then the $h^\ast$-polynomial of $\mathcal{O}(P,s)$ is real-rooted.
 \end{theorem}
 For the proof we adapt the interlacing techniques used by Savage and Visontai~\cite{sav_vis} and Brändén and Leander~\cite{lec_poset}.

 Analogous to the (unweighted) $h^\ast$-polynomial, the weighted $h^\ast$-polynomial of a lattice polytope can be defined as the numerator of the generating function of the weighted Ehrhart polynomial~\cite{weighted}. Furthermore, from the proof of Proposition~\ref{prop:weightlifting}, it follows that $P^\ast$ can be chosen to be a rooted tree for any rooted tree $P$. Theorem~\ref{thm:realrooted} thus also holds in the more general weighted setting of Proposition~\ref{prop:weightlifting}.
 \begin{corollary}\label{cor:weightedrealrooted}
        Let $P$ be a rooted tree on $[d]$, $s\colon [d]\rightarrow \mathbb{N}$ and let $w$ be any of the weights in Proposition~\ref{prop:weightlifting}. Then the weighted $h^\ast$-polynomial of $\mathcal{O}(P,s)$ is real-rooted.
\end{corollary}

\section{Proofs}
\subsection{Weight lifting polytopes}
In this subsection we prove Proposition~\ref{prop:weightlifting} by giving explicit constructions of $P^\ast$ and $s^\ast$. The idea is to realize any weight $w_i(x)$ that only depends on one coordinate $x(k)$ by attaching chains to the corresponding element $k$ of $P$. In particular, $P^\ast$ can be chosen to be an extension of $P$, and especially to be a rooted tree whenever $P$ is a rooted tree.

\begin{proof}[Proof of Proposition~\ref{prop:weightlifting}]
    We first assume that $\ell=1$ and $0\leq a<b$. In this case, we define the poset $P^\ast$ on $[d+b]$ as having the following relations:
    \[
    i\prec _{P^\ast} j \quad \text{ whenever } \quad i\prec _P j \text{ for all }i,j\in [d] \, , 
    \]
    and
    \[
    d+1 \prec _{P^\ast} d+2\prec _{P ^\ast}\cdots \prec _{P^\ast}d+a \prec _{P^\ast} d+b \prec _{P^\ast}d+b-1\prec _{P^\ast} \cdots \prec _{P^\ast} d+a+1 \prec _{P^\ast}k \, .
    \]
    We observe that $P^\ast$ arises from $P$ by attaching a (hanging) chain with $b$ elements to the element~$k$.
    In order to recover the weighted Ehrhart functions of $\mathcal{O}(P,s)$ as the (unweighted) Ehrhart polynomial of $\mathcal{O}(P^\ast,s^\ast)$ in the cases $(i)$ and $(iii)$, we define $s^\ast$ in two different ways. Let $\proj$ denotes the projection map from $\mathbb{R}^{d+b}$ onto the first $d$ coordinates.

    For $(i)$ we extend $s^\ast$ from $P$ to $P^\ast$ by setting
    $s^\ast (d+i)=s(k)$ for all $1\leq i\leq b$. Since $s$ is unchanged on $P$ we have that $\proj (\mathcal{O}(P^\ast,s^\ast))\subseteq \mathcal{O}(P,s)$. Thus, the Ehrhart function of $\mathcal{O}(P^\ast,s^\ast)$ can be calculated as
    \[
    \mathrm{E}_{\mathcal{O}(P^\ast,s^\ast)}(n)  \ = \ \sum _{x\in n\mathcal{O}(P,s)\cap \mathbb{Z}^d}|\proj ^{-1}(x)\cap n\mathcal{O}(P^\ast,s^\ast) \cap \mathbb{Z}^{d+b}| \, .
    \]
    By construction, for given $x$, $|\proj ^{-1}(x)\cap n\mathcal{O}(P^\ast,s^\ast) \cap \mathbb{Z}^{d+b}|$ equals the number of all integer vectors $(x(d+1),\ldots, x(d+b))$ for which 
    \[
0\leq \frac{x(d+1)}{s(k)}\leq \cdots \leq \frac{x(d+a)}{s(k)} \leq \frac{x(d+b)}{s(k)}<\frac{x(d+b-1)}{s(k)}<\cdots <\frac{x(d+a+1)}{s(k)}<\frac{x(k)}{s(k)}\, .
    \]
     By elementary counting this is number is equal to ${ x(k) +a \choose b}$ which proves $(i)$ in the case $0\leq a<b$.
    
    To prove the claim $(iii)$, we extend $s^\ast$ from $P$ to $P^\ast$ by setting
    $s^\ast (d+i)=1$ for all $1\leq i\leq b$. Then $|\proj ^{-1}(x)\cap n\mathcal{O}(P^\ast,s^\ast) \cap \mathbb{Z}^{d+b}|$ equals the number of all integer vectors $(x(d+1),\ldots, x(d+b))$ for which 
    \[
0\leq x(d+1)\leq \cdots \leq x(d+a) \leq x(d+b)<x(d+b-1)<\cdots <x(d+a+1)<x(k)/s(k) \, ,
    \]
    which is equal to $w_i(x) = \dbinom{\lceil x(k)/s(k)\rceil + a}{b}$ as desired.

Claim $(ii)$ and $(i)$ for $0<a\leq b$ can be seen analogously if we consider the poset $P^\ast$ on $[b+d]$ with the relations:
    \[
    d+i\prec _{P^\ast} d+j \quad \text{ whenver } \quad i\prec _P j \text{ for all }i,j\in [d] \, , 
    \]
    and
    \[
     b \prec _{P^\ast}b-1\prec _{P^\ast} \cdots \prec _{P^\ast} a+1 \prec _{P^\ast}1 \prec _{P^\ast} 2\prec _{P ^\ast}\cdots \prec _{P^\ast}a \prec _{P^\ast}d+k \, .
    \]
    Note that the shift of the elements in $P$ is necessary to cover the case $a=b$, while the extension $s^\ast$ is defined in the same way as above.

        By attaching multiple such chains (see \Cref{ex:monomial,ex:prodbinom}), we can also realize the products of such weights ($l \geq 2$) via lecture hall order polytopes.
\end{proof}

\begin{remark}
    The exact choice of elements in the attached chain in the poset $P^\ast$ is not important in the proof of Proposition~\ref{prop:weightlifting}. It is sufficient that the elements in the chain, together with the last element $k$, form a permutation with $a$ decents (after renormalization).
\end{remark}
\begin{example}\label{ex:monomial}
    If $d = 4$, the poset $P$ is the naturally ordered chain on the elements $[4]$, $w(x) = x(1) x(3)^2 x(4)$, and $(P^\ast, s^\ast)$ is as in \Cref{fig:monomial}, then $\mathrm{E}_{\mathcal{O}(P,s)}(w;n)=\mathrm{E}_{\mathcal{O}(P^\ast,s^\ast)}(n)$.
\end{example}

\begin{example}\label{ex:prodbinom}
    If $P$ is the poset on the left in \Cref{fig:binomial}, and the weight $w$ is given by
    \begin{equation*}
        w(x) = \binom{\lfloor x(2)/s(2)\rfloor + 1}{2} \cdot x(3) \binom{\lceil x(3)/s(3)\rceil}{2} \cdot (\lfloor x(4)/s(4)\rfloor + 1)^2 \cdot \binom{x(5) + 2}{3},
    \end{equation*}
    then we have $\mathrm{E}_{\mathcal{O}(P,s)}(w;n)=\mathrm{E}_{\mathcal{O}(P^\ast,s^\ast)}(n)$ where $(P^\ast, s^\ast)$ is given on the right in \Cref{fig:binomial}.
\end{example}

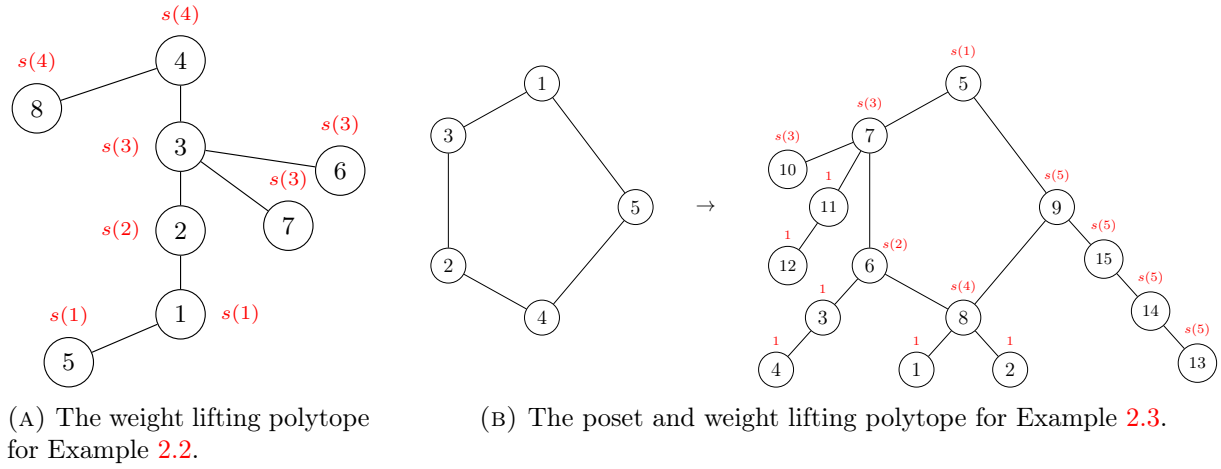
\begin{figure}[H]
    \centering

    \begin{subfigure}[t]{0.30\textwidth}
        \centering
        \resizebox{\linewidth}{!}{%
        \begin{tikzpicture}[yscale=0.6]
        \node[circle, draw, label = {\tiny \color{red} $s(4)$}] (4) at (0,4) {\footnotesize $4$};
        \node[circle, draw, label = {\tiny \color{red} $s(4)$}] (8) at (-1.8,3) {\footnotesize $8$};
        \node[circle, draw] (3) at (0,2.2) {\footnotesize $3$};
        \node at (-0.75, 2.2) {\tiny \color{red} $s(3)$};
        \node[circle, draw, label = {\tiny \color{red} $s(3)$}] (6) at (2.0,1.7) {\footnotesize $6$};
        \node[circle, draw, label = {\tiny \color{red} $s(3)$}] (7) at (1.35,0.55) {\footnotesize $7$};
        \node[circle, draw] (2) at (0,0.45) {\footnotesize $2$};
        \node at (-0.75, 0.45) {\tiny \color{red} $s(2)$};
        \node[circle, draw] (1) at (0,-1.3) {\footnotesize $1$};
        \node at (0.75, -1.3) {\tiny \color{red} $s(1)$};
        \node[circle, draw, label = {\tiny \color{red} $s(1)$}] (5) at (-1.4,-2.3) {\footnotesize $5$};

        \draw (5) -- (1) -- (2) -- (3) -- (7);
        \draw (6) -- (3) -- (4) -- (8);
        \end{tikzpicture}}
        \caption{The weight lifting polytope for \Cref{ex:monomial}.}
        \label{fig:monomial}
    \end{subfigure}
    \hfill
    \begin{subfigure}[t]{0.65\textwidth}
        \centering
        \resizebox{\linewidth}{!}{%
        \begin{tikzpicture}[xscale=0.9]
                    \node[circle, draw] (p4) at (4 - 9, 1) {4};
        \node[circle, draw] (p2) at (2 - 9, 2) {2};
        \node[circle, draw] (p3) at (2 - 9, 4.5) {3};
        \node[circle, draw] (p1) at (4 - 9, 5.5) {1};
        \node[circle, draw] (p5) at (6 - 9, 3.125) {5};
        \draw (p4) -- (p2) -- (p3) -- (p1) -- (p5) -- (p4);

        \node at (-1.5, 3.125) {$\rightarrow$};
        
        \node[circle, draw, label = {\tiny \color{red} $s(4)$}] (11) at (4, 1) {8};
        \node[circle, draw, label = {\tiny \color{red} $1$}] (9) at (3, 0) {1};
        \node[circle, draw, label = {\tiny \color{red} $1$}] (10) at (5, 0) {2};

        \node[circle, draw] (3) at (2, 2) {6};
        \node at (2.55, 2.4) {\color{red} \tiny $s(2)$};
        \node[circle, draw, label = {\tiny \color{red} $1$}] (2) at (1, 1) {3};
        \node[circle, draw, label = {\tiny \color{red} $1$}] (4) at (0, 0) {4};

        \node[circle, draw, label = {\tiny \color{red} $s(3)$}] (5) at (2, 4.5) {7};
        \node[circle, draw, label = {\tiny \color{red} $s(3)$}] (6) at (0.25, 3.85) {\footnotesize 10};
        \node[circle, draw, label = {\tiny \color{red} $1$}] (7) at (1.125, 3.125) {\footnotesize 11};
        \node[circle, draw, label = {\tiny \color{red} $1$}] (8) at (0.25, 2) {\footnotesize 12};

        \node[circle, draw, label = {\tiny \color{red} $s(1)$}] (1) at (4, 5.5) {5};

        \node[circle, draw, label = {\tiny \color{red} $s(5)$}] (15) at (6, 3.125) {9};
        \node[circle, draw, label = {\tiny \color{red} $s(5)$}] (14) at (7, 2.125) {\footnotesize 15};
        \node[circle, draw, label = {\tiny \color{red} $s(5)$}] (12) at (8, 1.125) {\footnotesize 14};
        \node[circle, draw, label = {\tiny \color{red} $s(5)$}] (13) at (9, 0.125) {\footnotesize 13};

        \draw (11) -- (3) -- (5) -- (1) -- (15) -- (11);
        \draw (9) -- (11) -- (10);
        \draw (4) -- (2) -- (3);
        \draw (8) -- (7) -- (5) -- (6);
        \draw (13) -- (12) -- (14) -- (15);
        \end{tikzpicture}}
        \caption{The poset and weight lifting polytope for \Cref{ex:prodbinom}.}
        \label{fig:binomial}
    \end{subfigure}

    \caption{Examples of weight lifting polytopes.}
    \label{fig:examples}
\end{figure}

In the proof of Proposition~\ref{prop:weightlifting} we attached chains to elements to construct the weight lifting polytope of the given binomial weights. Those chains are themselves $s$-lecture hall order polytopes. We conclude this section by observing that the same argument can be applied to obtain the following general weight lifting result by attaching posets $P_i$ to the elements $i$ of $P$.

\begin{proposition}\label{prop:ulecweights}
    Let $l \geq 0$ and let $w \colon \R^d \rightarrow \R$ be given by $w(x) = w_1(x) \cdots w_l(x)$ where for each $i \in [l]$, the function $w_i$ is of the form 
    \begin{itemize}
        \item[(i)] $w_i(x) = \mathrm{E}_{\mathcal{O}(P_i,u_i)}(x(i))$,
        \item[(ii)] $w_i(x) = \mathrm{E}_{\mathcal{O}(P_i)}(\lfloor x(i)/s(i)\rfloor)$, or
                \item[(iii)] $w_i(x) = \mathrm{E}_{\mathcal{O}(P_i)}(\lceil x(i)/s(i)\rceil -1)$
    \end{itemize}
    for finite posets $P_i$ and sequences of integers $u_i$. 
    Then there exists a poset $P^\ast$ and function $s^\ast : P^\ast \rightarrow \N$ such that \[\mathrm{E}_{\mathcal{O}(P,s)}(w;n)=\mathrm{E}_{\mathcal{O}(P^\ast,s^\ast)}(n) \, . \]
\end{proposition}

\subsection{Interlacing polynomials.}
Before going forward with the proof of \Cref{thm:realrooted}, we gather necessary preliminaries on real-rooted and interlacing polynomials. The reader familiar with the latter can skip this subsection.
Recall that the polynomial $f$ \textbf{interlaces} the polynomial $g$, denoted $f \preceq g$, if both $f$ and $g$ are real-rooted, and if the roots of $f$ and $g$ are $r_m \leq \cdots \leq r_1$ and $s_k \leq \cdots \leq s_1$ respectively, then
    \begin{equation*}
        \cdots \leq r_2 \leq s_2 \leq r_1 \leq s_1.
    \end{equation*}
Note that if $f \preceq g$, then $\deg f \leq \deg g \leq \deg f + 1$. Further, $0\preceq f$ and $f\preceq 0$ by definition. We will be working with sequences of interlacing polynomials.

\begin{lemma}{\cite[Lemma 2.3]{petter_polya}}\label{firstlast_interlacing}
    If $f_1, \ldots, f_m$ are real-rooted polynomials with $f_{i - 1} \preceq f_i$ for all $2 \leq i \leq m$ and $f_1 \preceq f_m$, then $f_i \preceq f_j$ whenever $i < j$.
\end{lemma}
The following result forms the backbone of many real-rootedness results in combinatorics in recent years as it allows to define sequences of interlacing polynomials recursively, see~\cite{unimodalitybeyond}. In particular, it was used in~\cite{sav_vis}.
\begin{proposition}\label{prop:interlacingsequences}
    Let $f_1, \ldots, f_m$ be real-rooted polynomials with nonnegative coefficients such that $f_i \preceq f_j$ for all $i < j$. 
    Let $0 \leq k_1 \leq \cdots \leq k_{m'} \leq m$ and for each $i \in [m']$, let
    \begin{equation*}
        g_i(t) = t \sum_{l = 0}^{k_i - 1} f_l(t) + \sum_{l = k_i}^m f_l(t).
    \end{equation*}
    Then $g_i \preceq g_j$ for all $i < j$.
\end{proposition}
    Given a polynomial $h$ of degree at most $d$, the \textbf{$f$-polynomial} of $h$ with respect to degree $d$, denoted $f = f(d, h)$, is defined by
    \begin{equation*}
        f(t) \coloneqq (1 + t)^dh\left(\frac{t}{1 + t}\right).
    \end{equation*}
In order to prove Theorem~\ref{thm:realrooted}, we will consider the $f$-polynomials of $h^\ast$-polynomials as they have desirable properties outlined in the following. 
In particular, for polynomials with non-negative coefficients, it is well-known that passing to the $f$-polynomials preserves real-rootedness and interlacing in the following sense.

\begin{lemma}\label{lem:fromhtof}
    Let $h$ be a polynomial of degree at most $d$ and $f = f(d, h)$. Then
    \begin{itemize}
        \item[(i)] The $h$ is real-rooted with roots in $(-\infty, 0]$ if and only if $f$ is real-rooted with roots in $(-1, 0]$.

        \item[(ii)] If $g$ is a polynomial of degree at most  $d'$, then $g \preceq h$ with both having roots in $(-\infty, 0]$ if and only if $f(d', g) \preceq f$ with both having roots in $(-1, 0]$.
    \end{itemize}
\end{lemma}

The $f$-polynomial also interacts well with respect to Hadamard products of formal power series. Given two formal power series $A, B$ defined by
\begin{equation*}
    A = \sum_{n \geq 0} a_n t^n \quad \text{ and } \quad B = \sum_{n \geq 0} b_n t^n,
\end{equation*}
their \textbf{Hadamard product} $A \ast B$ is the formal power series
\begin{equation*}
    A \ast B \coloneqq \sum_{n \geq 0} a_nb_n t^n.
\end{equation*}
If the coefficients sequence $a_n$ is given by the evaluation of a polynomial of degree at most $d$ in~$n$ then there exists a polynomial $h$ of degree at most $d$ such that
\[
A \ = \ \frac{h(t)}{(1-t)^{d+1}}\, .
\] See, e.g., \cite[Section 4.3]{enumcomb1}. We will consider the numerator polynomial of formal power series arising from counting lattice points in $s$-lecture hall polytopes under Hadamard products. The corresponding operation on $f$-polynomials is the diamond product. 
For polynomials $f, g$, their \textbf{diamond product} is defined by
\begin{equation*}
    (f \diamond g)(t) \coloneqq \sum_{i \geq 0} \frac{f^{(i)}(t)}{i!} \frac{g^{(i)}(t)}{i!} t^i (t + 1)^i
\end{equation*}
where $f^{(i)}(t)$ is the $i$-th order derivative of $f$. 
The relation between Hadamard products and diamond products that we use is as follows.

\begin{theorem}{\cite[Theorem 2.3]{wagner}}\label{thm:wagner}
    If $h, h_1$, and $h_2$ are polynomials of degrees at most $d, d_1$, and $d_2$ respectively, and
    \begin{equation*}
        \frac{h(t)}{(1 - t)^{d + 1}} = \frac{h_1(t)}{(1 - t)^{d_1 + 1}} \ast \frac{h_2(t)}{(1 - t)^{d_2 + 1}},
    \end{equation*}
    then we have $f(d, h) = f(d_1, h_1) \diamond f(d_2, h_2)$.
\end{theorem}

The proof of our main theorem relies on the preservation of the real-rootedness  and interlacing property under the diamond product described by the following results.

\begin{theorem}{\cite[Theorem 0.3]{wagner}}
    If two polynomials $f,g$ have all their roots in the interval $[-1,0]$ then so does $f\diamond g$.
\end{theorem}

\begin{theorem}{\cite[Theorem 12]{petter_rrns}}\label{diamond_interlacing}
    Let $f, g, h$ be real-rooted polynomials, and let $h$ have zeroes in $[-1, 0]$. 
    Then $f \diamond h$ is real-rooted, and if $f \preceq g$, then $f \diamond h \preceq g \diamond h$.
\end{theorem}

\subsection{Rooted trees.}
We can now move on to the proof of \Cref{thm:realrooted}, for which we make the following definitions. 
For the rest of this section, we let $T$ be a poset given by a rooted (hanging) tree on $[d]$ with root $v\in [d]$ and $s$ be a map $s \colon [d] \rightarrow \Z_+$. See Figure~\ref{fig:hangingtree} for an example. We consider the cone
\[
C_T \ := \ \left\{x\in \mathbb{R}_{\geq 0}^d \colon     \begin{aligned}
      & f(i)/s(i)\leq f(j)/s(j) \text{ whenever } i\prec j \, , \text{ and }\\
      & f(i)/s(i)<f(j)/s(j) \text{ whenever } i\prec j \text{ and } i>j
    \end{aligned} \right\}
\]
of all (real-valued) $(T,s)$-partitions. For given integers $0 \leq r < s(v)$, we consider the function 
        \begin{equation*}
            p_{T, r}(n) \coloneqq |\{x \in C_T \cap \Z^d \colon x(v) = ns(v) + r\} \, , \quad n\geq 0 \, .
        \end{equation*}
        The proof of Theorem~\ref{thm:realrooted} will follow from the following key result.
        \begin{theorem}\label{thm:interlacing}
            For all integers $r,r'$, $0 \leq r'<r < s(v)$,
            \begin{itemize}
                \item[(i)] the function $p_{T,r}(n)$ is given by a polynomial of degree $d-1$, and
                \item[(ii)] we have $A_{T, r} \preceq A_{T, r'}$, where 
                        \begin{equation*}
            \sum_{n \geq 0} p_{T, r}(n) t^n = \frac{A_{T, r}(t)}{(1 - t)^d}\, .
        \end{equation*}
            \end{itemize}           
        \end{theorem}
        In particular, $A_{T, r}$ is real-rooted. This will prove \Cref{thm:realrooted} by the following argument.\begin{proof}[Proof of~\Cref{thm:realrooted}]
            Set $T'$ to be the rooted tree obtained from $T$ by adding a root $d + 1$ above $v$. 
Define $s' : [d + 1] \rightarrow \Z_+$ by $s'(i) = s(i)$ for all $i \in [d]$ and $s(d + 1) = 1$. Then $p_{T',0}= |\{x\in C_{T'}\cap \mathbb{Z}^{d+1}\colon x(d+1)=n\}|=|n\mathcal{O}(T,s)\cap \mathbb{Z}^d|$ which gives $A_{T', 0} = h^\ast_{\mathcal{O}(T, s)}$.
        \end{proof}

\begin{remark}
    A combinatorial interpretation can be given to the polynomials $A_{T,r}$. Using the fact that $v$ is the maximal element of $T$, one can obtain an expression for $A_{T, r}$ as a special case of \cite[Theorem 3.2]{lec_poset}. 
    Using the same notations as in \cite{lec_poset}, we have
    \begin{equation*}
        A_{T, r}(t) = \sum_{\substack{\tau = (\pi, \tilde r) \in \mathcal{L}(T, s)\\ \tilde r(v) = r}} t^{|D_1(\tau)|}.
    \end{equation*}
\end{remark}

\begin{proof}[Proof of~\Cref{thm:interlacing}]   
    The proof is by induction on $d$. If $d = 1$ then $C_{T}$ is the positive real axis and $p_{T,r}\equiv 1$, thus the claim holds. Now assume $d >1$ and that the coatoms of $T$ are $v_1, v_2, \ldots, v_m$. 
    Let $T_1, T_2, \ldots, T_m$ be the sub-trees corresponding to them (i.e. lower ideals generated by the coatoms; see Figure~\ref{fig:subtrees.}) and $d_i$ be the number of elements of $T_i$. 
    The polynomials $p_{T_i, r}$ and $A_{T_i, r}$ are defined analogously.

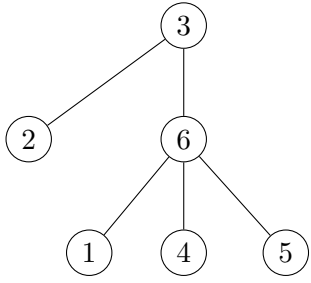
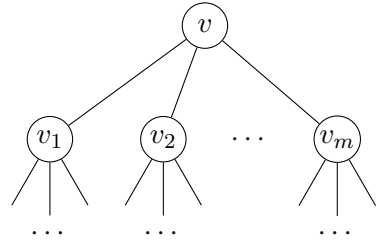
\begin{figure}
\begin{subfigure}[t]{0.45\textwidth}
\centering
        \begin{tikzpicture}[
            vertex/.style={
                circle,
                draw=black,
                minimum size=6mm,
                inner sep=1pt
            },
            dotnode/.style={draw=none}
        ]
        \node[vertex] (v)  at (-0.15,2.5) {$3$};
        
        \node[vertex] (v2) at (-2.2,1) {$2$};
        \node[vertex] (v6) at (-0.15,1) {$6$};

               \node[vertex] (v1) at (-1.4,-0.5) {$1$};
        \node[vertex] (v4) at (-0.15,-0.5) {$4$};
        \node[vertex] (v5) at (1.2,-0.5) {$5$};
        
        \draw (v) -- (v2);
        \draw (v) -- (v6);
        \draw (v6) -- (v1);
        \draw (v6) -- (v4);
        \draw (v6) -- (v5);
        \end{tikzpicture}
        \caption{Example of a poset given by a hanging tree.}
        \label{fig:hangingtree}
           \end{subfigure}
    \hfill
    \begin{subfigure}[t]{0.45\textwidth}
    \centering
        \begin{tikzpicture}[
            vertex/.style={
                circle,
                draw=black,
                minimum size=6mm,
                inner sep=1pt
            },
            dotnode/.style={draw=none}
        ]

        \node[vertex] (v)  at (-0.15,2.5) {$v$};
        
        \node[vertex] (v1) at (-2.2,1) {$v_1$};
        \node[vertex] (v2) at (-0.7,1) {$v_2$};
        \node[dotnode]     at (0.45,1) {$\cdots$};
        \node[vertex] (vm) at (1.6,1) {$v_m$};
        
        \draw (v) -- (v1);
        \draw (v) -- (v2);
        \draw (v) -- (vm);
        
        \foreach \x in {v1,v2,vm} {
            \draw (\x.240) -- ++(240:0.7);
            \draw (\x.270) -- ++(270:0.7);
            \draw (\x.300) -- ++(300:0.7);
        
            \path (\x) ++(0,-1.25) node[dotnode] {$\cdots$};
        }   
        \end{tikzpicture}
        \caption{Roots of subtrees (coatoms).}
        \label{fig:subtrees.}
    \end{subfigure}
    \caption{Rooted trees.}
    \label{fig:placeholder}
\end{figure}

    For given $0 \leq r' < s(v_i)$ and $0 \leq r < s(v)$, the inequality $(ks(v_i) + r')/s(v_i) \leq (ns(v) + r)/s(v)$ is equivalent to
    \[
    \begin{cases}
        k \leq n & \text{if }r'/s(v_i) \leq r/s(v)\, ,\\
        k < n & \text{if }r'/s(v_i) > r/s(v) \, .
    \end{cases}
    \]
    Similarly, the inequality $(ks(v_i) + r')/s(v_i) < (ns(v) + r)/s(v)$ is equivalent to
    \[
    \begin{cases}
        k \leq n & \text{if }r'/s(v_i) < r/s(v)\, ,\\
        k < n & \text{if }r'/s(v_i) \geq r/s(v) \, .
    \end{cases}
    \]
    Using the tree structure of $T$ we obtain
    \begin{eqnarray*}
        p_{T,r}(n)&=& \prod _{i\in [m]\atop v_i<v}|\{x\in C_{T_i}\cap \mathbb{Z}^d_i \colon \tfrac{x(v_i)}{s(v_i)}\leq n + \tfrac{r}{s(v)}\}| \cdot \prod _{i\in [m]\atop v_i>v}|\{x\in C_{T_i}\cap \mathbb{Z}^d_i \colon \tfrac{x(v_i)}{s(v_i)} < n + \tfrac{r}{s(v)}\}|\\
        &=& \prod _{i\in [m]\atop v_i<v} \left(\sum_{r'\colon r'/s(v_i) \leq r/s(v)} \sum _{k\leq n} p_{T_i,r'}(k) + \sum_{r'\colon r'/s(v_i) > r/s(v)}\sum _{k< n} p_{T_i,r'}(k)\right)\\
        &\times & \prod _{i\in [m]\atop v_i>v} \left(\sum_{r'\colon r'/s(v_i) < r/s(v)} \sum _{k\leq n} p_{T_i,r'}(k) + \sum_{r'\colon r'/s(v_i) \geq r/s(v)}\sum _{k< n} p_{T_i,r'}(k)\right)
    \end{eqnarray*}
    In particular, since $p_{T_i,r'}(k)$ is a polynomial of degree $d_i -1$  by induction it follows that $p_{T,r}$ is a polynomial of degree $d_1+d_2+\cdots +d_m=d-1$.

        It remains to show that $A_{T,r} \preceq A_{T,r'}$ whenever $0 \leq r' < r < s(v)$. Setting $k_i = \lfloor rs(v_i)/s(v)\rfloor$ for $i \in [m]$ such that $v_i < v$ and $k_i = \lceil rs(v_i)/s(v)\rceil - 1$ for $i \in [m]$ such that $v_i > v$, on the level of generating functions the equation above translates to
    \begin{equation*}
        \frac{A_{T, r}(t)}{(1 - t)^d} = \mathop{\mathlarger{\ast}}\limits_{i = 1}^m \left(\sum_{r' \leq k_i} \frac{A_{T_i, r'}(t)}{(1 - t)^{d_i + 1}} + \sum_{r' > k_i} \frac{tA_{T_i, r'}(t)}{(1 - t)^{d_i + 1}}\right).
    \end{equation*}
    For any $i \in [m]$ and $0 \leq r < s(v)$, we define
    \begin{gather*}
        f_i \coloneqq f\left(d_i, \sum_{r'} A_{T_i,r'}\right), \quad \tilde{f}_i \coloneqq f\left(d_i, t\sum_{r'} A_{T_i,r'}\right), \quad h_r \coloneqq f\left(d - 1, A_{T, r}\right),\text{ and}\\
        g_{i, r} \coloneqq f\left(d_i, \sum_{r' \leq k_i}A_{T_i, r'} + t\sum_{r' > k_i}A_{T_i, r'}\right).
    \end{gather*}

    By Lemma~\ref{lem:fromhtof} we need to show that $h_r \preceq h_{r'}$ whenever $0 \leq r' < r < s(v)$. From the computations above, for any $0 \leq r < s(v)$, we have $h_r = \mathop{\diamond}\limits_{i = 1}^m g_{i, r}$ by Theorem~\ref{thm:wagner}. 
    By induction and Proposition~\ref{prop:interlacingsequences}, for any $i \in [m]$ and $r' < r$, we have $f_i \preceq g_{i, r} \preceq g_{i, r'} \preceq \tilde{f}_i$.

    Let $0 \leq r' < r < s(v)$. 
    Repeatedly applying \Cref{diamond_interlacing} using $f_i \preceq g_{i, r}$, we get
    \begin{equation*}
        \mathop{\diamond}\limits_{i = 1}^m f_i\ \preceq\ g_{1, r} \diamond f_2 \diamond f_3 \diamond \cdots \diamond f_m\ \preceq\ g_{1, r} \diamond g_{2, r} \diamond f_3 \diamond \cdots \diamond f_m\ \preceq \cdots \preceq\ \mathop{\diamond}\limits_{i = 1}^m g_{i, r}.
    \end{equation*}
    Similarly, using $g_{i, r} \preceq g_{i, r'}$ and $g_{i, r'} \preceq \tilde{f}_i$, we get
    \begin{align*}
        &\mathop{\diamond}\limits_{i = 1}^m g_{i, r}\ \preceq\ g_{1, r'} \diamond g_{2, r} \diamond \cdots \diamond g_{m, r}\ \preceq \cdots \preceq\ \mathop{\diamond}\limits_{i = 1}^m g_{i, r'}, \text{ and}\\[0.2cm]
        &\mathop{\diamond}\limits_{i = 1}^m g_{i, r'}\ \preceq\ \tilde{f}_1 \diamond g_{2, r'} \diamond \cdots \diamond g_{m, r'}\ \preceq \cdots \preceq\ \mathop{\diamond}\limits_{i = 1}^m \tilde{f}_i.
    \end{align*}
    By \Cref{firstlast_interlacing}, if we show that $\mathop{\diamond}\limits_{i = 1}^m f_i \preceq \mathop{\diamond}\limits_{i = 1}^m \tilde{f}_i$, then we get
    \begin{equation*}
        h_r = \mathop{\diamond}\limits_{i = 1}^m g_{i, r}\ \preceq\ \mathop{\diamond}\limits_{i = 1}^m g_{i, r'} = h_{r'}.
    \end{equation*}
    Suppose $h$ is the polynomial of degree at most $d$ such that $\mathop{\diamond}\limits_{i = 1}^m f_i = f(d, h)$ and $\tilde{h}$ is such that $\mathop{\diamond}\limits_{i = 1}^m \tilde{f}_i = f(d, \tilde{h})$. 
    We have
    \begin{equation*}
        \frac{\tilde{h}(t)}{(1 - t)^{d}} = \mathop{\mathlarger{\ast}}\limits_{i = 1}^m \left(\frac{t\sum_r A_{T_i, r}}{(1 - t)^{d_i + 1}}\right) = t \cdot \mathop{\mathlarger{\ast}}\limits_{i = 1}^m \left(\frac{\sum_r A_{T_i, r}}{(1 - t)^{d_i + 1}}\right) = \frac{t \cdot h(t)}{(1 - t)^{d}}.
    \end{equation*}
    This shows that $h \preceq \tilde{h}$ and hence $\mathop{\diamond}\limits_{i = 1}^m f_i \preceq \mathop{\diamond}\limits_{i = 1}^m \tilde{f}_i$.
\end{proof}

We conclude this section by noting that Corollary~\ref{cor:weightedrealrooted} also extends to the weights in \Cref{prop:ulecweights}, provided each of the posets $P_i$ are rooted forests (disjoint unions of rooted trees). 
Further, we note that the conditions on $a$ and $b$ in \Cref{prop:weightlifting} are required to ensure real-rootedness for the weighted $h^\ast$-polynomials. 
For example, with $d = 1, s(1) = 2, w(x) = \binom{x(1) + 3}{2}$, the corresponding weighted $h^\ast$-polynomial is $t^3 - 3t^2 + 7t + 1$, which has a negative coefficient and is not real-rooted. 
Similar examples can be constructed for other types of weights.

\section{Concluding remarks}
We conclude by remarking on interpretations, variants and further consequences of the above results.

We start by noting that combinatorial interpretations can be given for the coefficients of the weighted $h^\ast$-polynomials we have considered in this paper. Given any weight function as in  \Cref{prop:weightlifting} or \Cref{prop:ulecweights}, the weighted $h^\ast$-polynomial agrees with the usual $h^\ast$-polynomial of the weight lifting polytope described in the proof of these propositions. This weight lifting polytope is a $s$-lecture hall order polytope and \cite[Corollary 3.7]{lec_poset} gives a general expression for the $h^\ast$-polynomial of such polytopes.

\Cref{thm:realrooted} can be formulated also for duals of rooted hanging trees: If $P$ is the dual of a rooted hanging tree on $[d]$ and $s \colon P \rightarrow \Z_+$ a labeling of its elements, then the lattice points in dilates of $\mathcal{O}_{P,s}$ are in bijection with the lattice points in dilates of $\mathcal{O}_{P^\ast,s^\ast }$; here, $P^\ast$ is the rooted tree defined by $i \prec_{P^\ast} j$ if and only if $d - i + 1 \succ_P d - j + 1$ and $s^\ast$ is the map defined by $s^\ast(i) = s(d - i + 1)$. 
 By considering the map on $\R^d$ that sends $x(i)$ to $ns(i) - x(i)$ for all $i \in [d]$, it can then be seen that $\mathrm{E}_{\mathcal{O}(P, s)}(n) = \mathrm{E}_{\mathcal{O}(P^\ast, s^\ast)}(n)$ for all $n \geq 1$. See also \cite[Theorem 3.9]{lec_poset}.

 Finally, since taking disjoint unions of posets corresponds to taking the Hadamard product of the  Ehrhart series of the corresponding $s$-lecture hall order polytopes, \Cref{thm:realrooted} carries over also to rooted forests by~\cite[Theorem 0.1]{wagner}. In particular, the unit cube in $\R^d$ can be viewed as $\mathcal{O}(A, s)$ where $A$ is the antichain on $[d]$ (which is a forest) and $s(i) = 1$ for all $i \in [d]$. 
This shows that the weighted $h^\ast$-polynomial of the cube associated to any weights of the form described in \Cref{prop:weightlifting,prop:ulecweights} are real-rooted. 
This complements the results for unit cubes in \cite[Section 4.6]{Garzon_Mora2025-hk} that were obtained by the first author together with Jes\'us De Loera, Alexey Garber, Sof\'ia Garz\'on Mora and Josephine Yu.

\textbf{AI usage declaration:} The results, proofs and final write-up where produced without the usage of any AI, with the exception of help with the TikZ pictures.

\textbf{Acknowledgements:} KJ was partially supported by the Wallenberg AI, Autonomous Systems and Software program funded by
the Knut and Alice Wallenberg foundation, grant nr 2023-04063 from the Swedish research council and the Verg
Foundation. KM was partially supported by the Verg foundation and the Göran Gustafsson foundation.

\bibliographystyle{siam}
\bibliography{refs}

\end{document}